\documentclass{amsart}

\usepackage{amsmath,amsfonts,amsthm,amssymb,amscd}
\usepackage{hyperref}

\usepackage[alphabetic]{amsrefs}
\usepackage{tikz}
\usepackage{graphicx}
\usepackage{tikz-cd}
\usepackage{xcolor}

\usepackage{chngcntr}
\usepackage{apptools}

\usepackage[shortlabels]{enumitem}

\usepackage{geometry}
\usetikzlibrary{decorations.markings}
\usetikzlibrary{patterns}

\newtheorem{theorem}{Theorem}[section]
\newtheorem{lem}[theorem]{Lemma}
\newtheorem{prop}[theorem]{Proposition}
\newtheorem{thm}[theorem]{Theorem}
\newtheorem{cor}[theorem]{Corollary}

\theoremstyle{definition}
\newtheorem{deff}[theorem]{Definition}
\theoremstyle{question}
\newtheorem{ques}[theorem]{Question}

\newtheorem{conj}[theorem]{Conjecture}

\theoremstyle{remark}
\newtheorem{example}[theorem]{Example}
\newtheorem{remark}[theorem]{Remark}

\numberwithin{equation}{section}

\newcommand{\G}{\Gamma}

\begin{document}

\tikzset{->-/.style={decoration={
  markings,
  mark=at position .65 with {\arrow{>}}},postaction={decorate}}}

\title[Decomposing groups by codimension-1 subgroups]{Decomposing groups by codimension-1 subgroups}



\author{Nansen Petrosyan}
\address{School of Mathematics, University of Southampton, Southampton SO17 1BJ, UK}
\email{n.petrosyan@soton.ac.uk}


\date{\today}


\begin{abstract}  The paper is concerned with Kropholler's conjecture on splitting a finitely generated group over a codimension-$1$ subgroup.  For a subgroup $H$ of a group $G$, we introduce the notion of `finite splitting height'  which  generalises the finite-height property. By considering  the dual CAT(0) cube complex associated to a codimension-$1$ subgroup $H$ in $G$, we show that the Kropholler-Roller conjecture holds when $H$ has finite splitting height in $G$.     Examples of subgroups of finite height are stable subgroups or more generally strongly quasiconvex subgroups.  Examples of subgroups of finite splitting height include relatively quasiconvex subgroups of relatively hyperbolic groups with virtually polycyclic peripheral subgroups. In particular, our results extend Stallings' theorem and generalise a theorem of Sageev on decomposing a hyperbolic group  by quasiconvex subgroups.\end{abstract}

\maketitle

\setlist[enumerate,1]{before=\itshape,font=\normalfont, leftmargin=1.1cm}

\section{Introduction}\label{sec:intro}

The number of ends $e(G)$ is an important invariant of a finitely generated group $G$. By fundamental work of Stallings \cite{Stall1, Stall2}, a finitely generated group $G$ splits as a nontrivial free product with amalgamation or as an HNN-extension over a finite subgroup if and only if $e(G)>1$.  Houghton in \cite{Hough} generalised the number of ends of a group by defining $e(G,H)$ in particular for the pair of  a finitely generated group $G$ and its subgroup $H$. There are number of important generalisations of Stallings' theorem when $e(G, H)>1$ \cite{Bowd}, \cite{DunSag}, \cite{DunSw},  \cite{Krop},  \cite{Scott}, \cite{ScSw}, \cite{Swar}.

One can formulate Stallings' theorem, by stating that $G$ splits nontrivially over a finite group if and only if $G$ has a proper almost-invariant subset. In 1988, Kropholler proposed the following generalisation of Stallings' theorem \cite{Dun:fin}, \cite{NibSag}. 

\begin{conj}[Kropholler]\label{conj:K} Let $G$ be a finitely generated group and let $H$ be a subgroup. If $G$ contains an $H$-proper $H$-almost invariant subset $A$ such that $AH=A$, then $G$ admits a nontrivial splitting over a subgroup commensurable with a subgroup of $H$.
\end{conj}

\noindent  In \cite{Dun:fin}, Dunwoody outlined a proof of the conjecture, but his arguments in Section 3 of  \cite{Dun:fin} were later found to contain a gap \footnote{A personal communication with Martin Dunwoody and Alex Margolis, 2018.}. 
Kropholler's conjecture  is known  to hold  when either:
\begin{itemize}
\item[(i)]  $H$ and $G$ are Poincar\'{e} duality groups \cite{KR1, KR2, KR3, KR4}, 
\item[(ii)] $H$ is a virtually polycyclic group \cite{DunRoll}, 
\item[(iii)] $H$ is  finitely generated and commensurated in $G$, i.e.~$Comm_G(H)=G$ \cite{DunRoll}. 
\end{itemize}
\noindent We will give a new proof of Dunwoody-Roller's result (iii)  above which also covers the case when $H$ is not finitely generated.

According to Scott \cite{Scott}, $e(G, H)>1$ if and only if there exists an $H$-proper $H$-almost invariant subset $A$ of $G$ such that $HA=A$. Crucially for us, this is equivalent to the existence of an essential and in particular fixed-point-free action of $G$  on a CAT(0) cube complex where $H$ stabilises a hyperplane \cite{S1} (see Section \ref{sec:cubing}  and Remark \ref{rmk:fix}). Motivated by such geometric interpretations of relative ends, $H$ is said to have {\it codimension-1 in $G$} if $e(G, H)>1$. In the special case when $G$ acts simplicially on a tree without a fixed point, it has more than one end  relative to an edge stabiliser. The converse is not true however, since  for example many Coxeter groups \cite{Serre}, \cite{NiRe2} and Thompson's groups $T$, $V$ \cite{Far}  have Property (FA) and yet they act essentially on a CAT(0) cube complex.  So the existence of a codimension-1 subgroup does not guarantee a nontrivial splitting of $G$. Nonetheless, according to Kropholler's conjecture, if the associated subset $A$ is also right $H$-invariant, then $G$ splits nontrivially over a subgroup commensurable with a subgroup of $H$. This was conjectured by Kropholler and Roller in \cite{KR3}. 

\begin{conj}[Kropholler-Roller, \cite{KR3}]\label{conj:KR} Let $G$ be a finitely generated group and let $H$ be a subgroup. If $G$ contains an $H$-proper $H$-almost invariant subset $A$ such that $HAH=A$, then $G$ admits a nontrivial splitting over a subgroup commensurable with a subgroup of $H$.
\end{conj}

\noindent Of course, any group that satisfies Kropholler's conjecture also satisfies Conjecture \ref{conj:KR}. It is worth pointing out that Conjecture \ref{conj:KR} implies Kropholler's conjecture. This is because if there is an $H$-proper $H$-almost invariant subset $A$ satisfying $AH=A$, then there is an $L$-proper $L$-almost-invariant subset $B$ such that $LBL=B$ for a subgroup $L$ of $H$ \cite{DunRoll}.

\begin{deff}\label{def:sh} {\rm We say that  a subgroup $H$ is {\it splitting-compatible in $G$} if it satisfies Conjecture \ref{conj:KR}, that is, if $G$ contains an $H$-proper $H$-almost invariant subset $A$ such that $HAH=A$, then $G$ admits a nontrivial splitting over a subgroup commensurable with a subgroup of $H$.}
\end{deff}

\begin{remark}\label{remark:splitting-comp} By definition, if $G$ splits over a subgroup $H$, then $H$ is splitting-compatible in $G$. 
\end{remark}

\begin{thm}\label{thmA} Let $H$ be a subgroup of $G$.  Suppose there exists an integer $n>0$ such that for any $k>n$ distinct cosets $g_iH$, the intersection $\cap_{i=1}^{k} g_iH{g_i}^{-1}$  is splitting-compatible in $G$. Then $H$ is splitting-compatible in $G$. 
\end{thm}

 Let $H$ be a subgroup of  $G$. Recall that the height of $H$ in $G$, denoted $height_G(H)$, is the least integer $n$ such that for any $n+1$ distinct cosets $g_iH$, the intersection $\cap_{i=1}^{n+1}g_iH{g_i}^{-1}$ is finite. If such  $n$ exists, then $H$ is said to have finite height equal to $n$ in $G$. When $G$ is hyperbolic and $H$ is  quasiconvex in $G$, then $H$ has finite height in $G$, it has finite index in the stabiliser $H'$ of its limit set and $Comm_G(H')=H'$. Using these properties, Sageev in \cite{S2}  showed that Conjecture \ref{conj:KR} holds when $G$ is hyperbolic and $H$ is a quasiconvex subgroup in $G$.   We obtain the following generalisation of this result. 
\begin{cor}\label{cor:all} Let $G$ be a finitely generated group. 
\begin{enumerate}[label=($\roman*$)]
\item If $H$ is a subgroup of finite height in $G$, then Conjecture \ref{conj:KR} holds.
\item If $H$ is a relatively quasiconvex subgroup of a relatively hyperbolic group $G$ with splitting-compatible parabolic subgroups, then Conjecture \ref{conj:KR} holds.
\item If $H$ is a relatively quasiconvex subgroup of a relatively hyperbolic group $G$ with virtually polycyclic peripheral subgroups, then Conjecture \ref{conj:KR} holds.
\end{enumerate}
\end{cor}

\noindent Stable subgroups \cite{DuTa}  or more generally strongly quasiconvex subgroups \cite{Tran} are natural generalisations of quasiconvex subgroups of hyperbolic groups.   They satisfy similar properties such as finite height, width and bounded packing \cite{AMST}, \cite{Tran}. Stable subgroups are precisely those strongly quasiconvex subgroups that are hyperbolic \cite[Theorem 4.8]{Tran}. Examples of stable subgroups are:
\begin{itemize}
\item[-] convex cocompact subgroups of the mapping class group of a connected, orientable surface ${\rm Mod}(S)$,
 \item[-] convex cocompact subgroups of the outer automorphism group ${\rm Out}(F_n)$ of the free group on $n\geq 3$ generators,
 \item[-] finitely generated subgroups of a right-angled Artin group quasi-isometrically embedded in the extension graph.
 \end{itemize}
 Examples of strongly quasiconvex subgroups also include:
\begin{itemize}
\item[-] peripheral subgroups of a relatively hyperbolic groups,
 \item[-]   hyperbolically embedded subgroups of finitely generated groups.
\end{itemize}
So by Corollary \ref{cor:all} (i), all of the above subgroups are splitting-compatible in their respective ambient groups, i.e.~they satisfy Conjecture \ref{conj:KR}.

\begin{cor}\label{corT} Let $G$ be a finitely generated group. Suppose  $H$ is a finite-height codimension-1 subgroup.  If every action of $H$ on a CAT(0) cube complex has a global fixed point, in particular, if $H$ satisfies Kazhdan's property (T), then $G$ splits over a subgroup commensurable with a subgroup of $H$.
\end{cor}

Corollary \ref{corT} is analogous to  a theorem  of Kar and Niblo \cite[Theorem 2]{KN} of splitting a Poincar\'{e} duality group over a codimension-1 property (T) Poincar\'{e} duality subgroup. Our next two applications are motivated by a result of Sageev   \cite[Corollary 4.3]{S2} which gives a decomposition of a hyperbolic group by a chain  of successively codimension-1 quasiconvex subgroups. 

\begin{cor}\label{correl} Let $G$ be a relatively hyperbolic group with virtually polycyclic peripheral subgroups.  Suppose  $H$ is a codimension-1,   relatively quasiconvex subgroup of $G$. Then either $G$ splits nontrivially over a subgroup commensurable with a subgroup of $H$ or there exist  
$$H=H_1> H_2>\dots >H_{k+1},$$
where $g_i\in G$, $H_{i+1}=H_i\cap g_iH_ig_i^{-1}$ is a codimension-$1$, relatively quasiconvex subgroup of $H_{i}$, and $H_{k}$ splits nontrivially over a subgroup commensurable with a subgroup of $H_{k+1}$.\end{cor}

\begin{cor}\label{corstr}  Let $G$ be a finitely generated group. Suppose  $H$ is a codimension-1,   strongly quasiconvex subgroup of $G$. Then either $G$ splits nontrivially over a subgroup commensurable with a subgroup of $H$ or there exist  
$$H=H_1> H_2>\dots >H_{k+1},$$
where $g_i\in G$, $H_{i+1}=H_i\cap g_iH_ig_i^{-1}$  is a codimension-$1$, strongly quasiconvex subgroup of $H_{i}$, and $H_{k}$ splits nontrivially over a subgroup commensurable with a subgroup of $H_{k+1}$.\end{cor}

When $G$ is hyperbolic and $H$ is a codimension-1 quasiconvex subgroup, then the action of $G$ on a dual CAT(0) cube complex is cocompact \cite{S2}. Finding suitable codimension-1 subgroups in $G$ is generally the first step in constructing a proper and cocompact action of $G$ on a CAT(0) cube complex and thus cocompactly cubulating $G$ \cite{BerWise}. Such cocompact cubulations have been instrumental in understanding properties of hyperbolic 3-manifold groups and resolving the Virtual Haken Conjecture by the work of Agol \cite{Agol} and Haglund-Wise \cite{HagWise}.

\subsection*{Methods involved} In \cite{S1}, Sageev gave a new characterisation of relative ends. According to Theorem 3.1 of \cite{S1}, $e(G,H)>1$ if and only if  $G$ acts essentially on the dual  CAT(0) cube complex $X_A$ with one orbit of hyperplanes such that $H$ has finite index in a hyperplane stabiliser. 
We define a cubical subcomplex $C_A\subseteq X_A$ which we call the Cayley subcomplex of $X_A$. The group $G$ acts cocompactly on $C_A$. We show that when $H$ is commensurated in $G$, then the intersections of the hyperplanes of $X_A$ with $C_A$ are compact, implying that $C_A$ has more than one end. By applying  Dunwoody's theorem  \cite{Dun:cut} to $C_A$, we obtain Theorem \ref{thm:comm} which extends Dunwoody's and Roller's theorem from \cite{DunRoll}. 

We then consider the splitting obstruction $S_A(G, H)$ for the triple $(G, H, A)$ introduced by Niblo in \cite{Niblo}. The splitting obstruction is a subset of $G$ that measures to some extent the failure of $X_A$ to be a tree.  If $S_A(G, H)$ is `controlled', then one can still obtain a splitting of $G$ over a subgroup commensurable with $H$ \cite{Niblo}.  We prove that if  $S_A(G, H)$ stays outside of the controlled parameters, then Kropholler's corner argument can be applied repeatedly to show that there are sufficiently many distinct cosets $g_iH$ so that the intersection $L=\cap g_iH{g_i}^{-1}$ is splitting-compatible as well as there exists $L$-proper, $L$-almost invariant subset of $G$. This connection between the splitting obstruction and Kropholler's corner argument is the key new input that allows us to generalise Theorem 4.2  of Sageev in \cite{S2}. Hence several aforementioned properties of quasiconvex subgroups of hyperbolic groups  essential in Sageev's proof are no longer needed  in the proof of Theorem \ref{thmA}. Corollary \ref{cor:all} follows from Theorems \ref{thmA} by using the fact that finite or more generally virtually polycyclic subgroups are splitting-compatible.  We apply Corollary \ref{cor:all} (i) in the case when $H$ has Kazhdan's property (T) and derive Corollary \ref{corT}.

 In \cite{HW1}, Hruska and Wise, conjectured that any subgroup of a virtually polycyclic group has bounded packing. This was later verified by Yang \cite{Yang}. Combining  this result with a theorem of Hruska and Wise on bounded packing of relatively hyperbolic groups, one obtains that any relatively quasiconvex subgroup $H$ of a relatively hyperbolic group $G$ with virtually polycyclic peripheral subgroups has bounded packing in $G$. We then use Sageev's observation that when $H$ is codimension-$1$ and has bounded packing in $G$,  the dual cube complex  can be constructed to be finite dimensional. Corollary \ref{correl} is then derived using Corollary  \ref{cor:all}  (iii). The proof of Corollary \ref{corstr} is analogous to the proof of Corollary \ref{correl}.

\subsection*{Acknowledgments}  I would like to thank Martin Dunwoody, Peter Kropholler, Alex Margolis, Armando Martino,  Ashot Minasyan and Graham Niblo for many helpful comments and discussions.

\section{Dual cube complex and Cayley subcomplex}\label{sec:cubing}

Given a finitely generated group $G$ and  a subgroup $H \leq G$, let $\G_H(G, S)$ denote the Schreier coset graph, i.e.~the quotient of the Cayley graph of $G$ by the action of $H$.  Houghton \cite{Hough} introduced  the number of ends of the pair of groups $(G,H)$ as $e(G,H) = e(\G_H(G, S))$. As in the classical case, the number of relative ends  does not depend on the finite generating set $S$ of $G$. Differently from $e(G)$ however, the number of relative ends  can take any nonnegative integer value \cite{Scott}.  

A subset of $G$ is said to be {\it $H$-finite}, if it is contained in finitely many right cosets of $H$. A subset $A$ of $G$ is {\it $H$-proper}  if $A$ and its complement $A^*$ are not $H$-finite. $A$ is said be {\it $H$-almost invariant} if the symmetric difference $A+ Ag$ is $H$-finite,  $\forall g\in G$.   One has that $e(G,H)>1$ if and only if there exists an $H$-proper $H$-almost invariant subset $A\subset G$ such that $HA=A$. According to Theorem 1.2 in \cite{NibSag},  $e(G, H)>1$ if and only if there is a fixed-point-free action of $G$ on a CAT(0) cube complex $X$ with one orbit hyperplanes where $H$ is a hyperplane stabiliser.

We refer the reader to \cite{HW2}, \cite{Sag:intro} for the definition of a CAT(0) cube complex and its basic properties some of which we now recall. Let $X$ be a CAT(0) cube complex. An {\it (oriented) combinatorial hyperplane} of $X$ is an equivalence class of (oriented) edges where the equivalence relation is generated by identifying opposite edges of any square in $X$. Given a combinatorial hyperplane $J$, one defines the associated geometric hyperplane as the union of all the dual blocks whose vertices are the midpoints of the edges in $J$. When there is no confusion, we will use the term {`hyperplane'} when referring to either combinatorial or geometric hyperplanes. The {\it carrier} of the combinatorial hyperplane $J$ is the is the union of all the cubes in $X$ that contain an edge of $J$.  Two distinct hyperplanes $I$ and $J$ are said to be {\it transverse} if the associated  geometric hyperplanes intersect. A  hyperplane $J$ partitions $X$ into two connected components $X^{\pm}$ called half-spaces. 

We will always assume that an action of a group on a cube complex is by automorphisms. An  action of  a group $G$ on a CAT(0) cube complex $X$ is said to be {\it essential} if there is a hyperplane $J$ of $X$ and a vertex $v\in X$ such that both $B=\{g\in G \; | \;  gv\in X^{+}\}$  and $B^*=\{g\in G \; | \;  gv\in X^{-}\}$ are not $Stab(J)$-finite.  In this case, $B$ is also $Stab(J)$-almost-invariant. The converse also holds.

\begin{thm}[Sageev \cite{S1}]\label{thm:Sageev} Let $G$ be a finitely generated group and $H$ be a subgroup. If $A$ is an $H$-proper $H$-almost invariant subset of $G$ and $HA=A$, then, there is a CAT(0) cube complex $X_A$ on which $G$ acts essentially with one orbit of hyperplanes such that
\begin{enumerate}[($i$)]
\item $H$ has finite index in the stabiliser of the hyperplane $J=(A, A^*)$ in $X_A$.
\item There is a natural correspondence $$gJ\Leftrightarrow (gA, gA^*), \; \forall g\in G.$$
\item Distinct hyperplanes $g_1J$ and $g_2J$ are transverse if and only if $g_1A$ and $g_2A$ are not nested. 
\item If $\Sigma_A$ has finite width, then $X_A$ is finite dimensional.
\item If $\Sigma_A$ is nested, i.e. of width $1$, then $X_A$ is a tree.
\end{enumerate}
\end{thm}
\noindent Here,  
$$\Sigma_A=\{ gA \; | \; g\in G\}\cup \{gA^* \; | \; g\in G\},$$ 
is the poset of translates of $A$ partially ordered by inclusion. Two elements $B_1, B_2\in \Sigma_A$ are said to be {\it nested} if one of the intersections
$B_1\cap B_2$, $B_1\cap B^*_2$, $B^*_1\cap B_2$ or $B^*_1\cap B^*_2$ is empty. The {\it width} of $\Sigma_A$ is the cardinality of the largest set of pairwise non-nested elements of $\Sigma_A$.

\begin{remark}\label{rmk:fix} From the definition of an essential action, it is straightforward to see that $G$ acts without a global fixed point on $X_A$. In fact the converse also holds, as was later shown by Gerasimov \cite{Ger} and Niblo-Roller \cite{NibRoll}, implying that $G$ acts essentially on a CAT(0) cube complex $X$ with one orbit of hyperplanes if and only if $G$ acts on $X$ without a global fixed point. 
\end{remark}

The  complex $X_A$ is called the {\it dual cube complex} associated to the triple $(G, H, A)$. In order to define a subcomplex in $X_A$, we briefly recall its construction from \cite{S1}.  Define $V\subset \Sigma_A$  to be a vertex if it satisfies the conditions:
\begin{itemize} 
\item[-] $\forall B\in \Sigma_A$, either $B$ or $B^*$ are in $V$ but not both.
\item[-] If $B_1\in V$, $B_2\in \Sigma_A$ and $B_1\subseteq B_2$, then $B_2\in V$.
\end{itemize}
For $g\in G$, a principal vertex is defined as $V_g=\{ B\in \Sigma \; | \; g\in B\}$. Define $\G$ as the graph with the vertex set $V\in \Sigma$ where any two vertices $V, W\subset \Sigma$ are connected by an edge $e$ if and only if there is $B\in \Sigma$ so that $W=(V\smallsetminus \{B\} )\cup \{B^*\}$. In which case the edge $e$ (with the induced orientation) is said to {\it exit} $B$. 
One then restricts to the connected component  $\G'$  of the graph containing all the principal vertices. The complex $X_A$ is then defined by attaching all possible $n$-cubes to $\G'$ for each $n\geq 2$ so that $\G'$ becomes the 1-skeleton of $X_A$.

A key part of the construction is to show that any two principal vertices in $\G$ are connected by an edge path. To this end, in \cite[Lemma 3.4]{S1}, it was shown that $|V_{x_1}\smallsetminus V_{x_2}|<\infty$, for any $x_1, x_2\in G$.  Recall that the {\it interval} $[V_{x_1}, V_{x_2}]$ from $V_{x_1}$ to $V_{x_2}$ is the convex hull of $V_{x_1}$ and $V_{x_2}$ which is the full subcomplex of $X_A$ formed by all the geodesic edge paths joining $V_{x_1}$ and $V_{x_2}$.

\begin{lem}\label{lem:path} For any $g, x_1, x_2\in G$, we have
\begin{enumerate}[($i$)]
\item  Let $p$ be a geodesic edge path from $V_{x_1}$ to $V_{x_2}$. Then every edge of $p$ exits some element of $V_{x_1}\smallsetminus V_{x_2}$. 
\item $[V_{gx_1}, V_{gx_2}]=g[V_{x_1}, V_{x_2}]$.

\end{enumerate}
\end{lem}
\begin{proof} For each given edge path $p$ joining $V_{x_1}$ to $V_{x_2}$, every $B\in V_{x_1}\smallsetminus V_{x_2}$ is exited by an edge of $p$. Following the proof of Theorem 3.3 of \cite{S1}, we can order
$$V_{x_1}\smallsetminus V_{x_2} = \{A_1, \dots, A_n\},$$
 so that each $A_j$ is minimal in $\{A_j, \dots, A_n\}$ which implies that $V_{x_1}\smallsetminus \{A_i\}_{i=1}^{j}\cup \{A^*_i\}_{i=1}^{j}$ are vertices connecting $V_{x_1}$ and $V_{x_2}$ by a geodesic edge path of length $n$.  If $p$ is geodesic, i.e.~of length $n$, we conclude that it must only consist of edges that exit  elements of $V_{x_1}\smallsetminus V_{x_2}$. This gives us part (i).
 
 Part (ii) follows from the fact that $gV_x=V_{gx}$, $\forall g, x\in G$.
\end{proof}

\begin{deff}\label{Cayley} Fix  a finite symmetric generating set $S$ of $G$. For each $g\in G$ and $s\in S$, join $V_g$ and $V_{gs}$ by intervals $[V_g, V_{gs}]$ in $X_A$. We define the  subcomplex $C_A =\cup_{s\in S, g\in G} [V_g, V_{gs}]$  of $X_A$ and call it the {\it Cayley complex}  associated to the triple $(G, S, A)$. 
\end{deff}

\noindent  Define an {\it (oriented) panelling} of $C_A$  to be  the intersection of an (oriented) hyperplane of $X_A$ with $C_A$. Hence, each panelling $\mathcal J = J\cap C_A$ separates $C_A$ into disconnected components.  Note that for any hyperplane $J$ of $X_A$, we have $Stab(J)=Stab(J\cap C_A)$. Also, $G$ acts cocompactly on $C_A$ with a single orbit of panellings.

Recall that 
$$Comm_G(H)=\{ g\in G \; | \; |H:H\cap gHg^{-1}|<\infty \mbox{ and }  |gHg^{-1}:H\cap gHg^{-1}|<\infty \}.$$
$Comm_G(H)$ is a subgroup of $G$ and $H$ is said to be {\it commensurated } in $G$ if $Comm_G(H)=G$.

\begin{figure}
\includegraphics[scale=0.5]{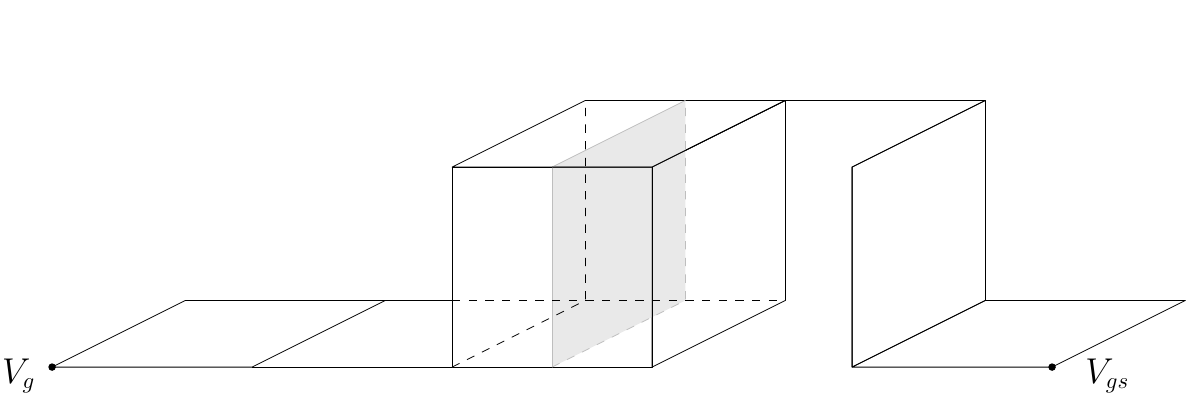}
\caption{The interval $[V_g, V_{gs}]$ and its intersection with the hyperplane $J$ (shaded).}
\label{fig1}
\end{figure} 

\begin{prop}\label{prop:comm} Suppose $H$ is a commensurated subgroup of a finitely generated group $G$. If $A$ is an $H$-proper $H$-almost invariant subset of $G$ and $HAH=A$, then the panellings of the Cayley complex $C_A\subseteq X_A$ are compact. In particular, $G$ splits nontrivially over a subgroup commensurable with $H$. 
\end{prop}
\begin{proof}  Consider the Cayley complex $C_A\subseteq X_A$. Let $J=(A, A^*)$ with stabiliser $Stab(J)$ containing $H$ as a finite index subgroup. Consider the panelling $\mathcal J=J\cap C_A$. To show that $\mathcal J$ is compact, it suffices to show that there are only finitely many intervals intersecting the hyperplane $J$. So, let $[V_g, V_{gs}]$ be an interval intersecting $J$ (see Figure \ref{fig1}). This means that either $g\in A$ and $gs\in A^*$ or $g\in A^*$ and $gs\in A$. Since $S$ is symmetric, without loss of generality, we can assume the former. But this means that $g\in A\cap A^*s^{-1}$. Since $A$ is $H$-almost invariant, the set  $A\cap A^*s^{-1}$ is $H$-finite and only depends on the choice of $s\in S$. So there exists a finite subset $F\subset G$, such that  $g\in HF$. Since $H$ is commensurated in $G$, there are finite subsets $E,K\subset G$ such that $HF\subseteq EH$ and $EHS\subseteq KH$. Since $V_{x}=V_{xh}=xhV_{1}$, $\forall h\in H$, $\forall x\in G$, we get 
$$\cup_{s\in S, g\in A\cap A^*s^{-1}} [V_g, V_{gs}]\subseteq \cup_{s\in S, g\in EH}[V_g, V_{gs}] \subseteq \cup_{x\in E, y\in K}[V_x, V_y],$$
which shows that there are finitely many such intervals. 

Now, consider the $1$-skeleton $C_A^1$ of $C_A$. The edges of each hyperplane $\mathcal J$ of $C_A$ form a finite cut of $C_A^1$. It follows that $C_A^1$ has more than one end and $G$ acts on $C_A^1$ with unbounded orbits. We can apply Theorem 4.1 of Dunwoody \cite{Dun:cut}, to obtain a nontrivial splitting over a subgroup which is a finite extension of $Stab(\mathcal J)=Stab(J)$.\end{proof}

\begin{lem}[{\cite[Step 1]{DunRoll}}]\label{lem:binv} Let $H$ be a subgroup of a finitely generated group $G$ and $A$ be an $H$-proper $H$-almost invariant subset of $G$ satisfying $AH=A$. Then there exists a subgroup $L\leq H$ and a subset $B\subset G$ so that $LBH=B$ and $B$ is $L$-proper $L$-almost invariant. Moreover, if $H$ is commensurated in $G$, then $[H:L]<\infty$.
\end{lem}

\begin{thm}\label{thm:comm} Let $H$ be a commensurated subgroup of a finitely generate group $G$. If there is an $H$-proper $H$-almost invariant subset $A$ such that $AH=A$, then $G$  splits nontrivially over a subgroup commensurable with  $H$.\end{thm}

\begin{proof} By Lemma \ref{lem:binv}, let $L$ be a finite index subgroup of $H$ and $B$ be an $L$-proper $L$-almost invariant subset such that $LBL=B$. Then apply Proposition \ref{prop:comm}.
\end{proof}

\section{Subgroups of finite splitting height}\label{sec:fch}

For a subgroup $H$ of a finitely generated group $G$ consider the natural action of $G$ on the set of  left cosets $G/H$.  Let $X\subseteq G/H$. Define $Stab(X)$ as the subgroup of $G$ that leaves $X$ invariant and let $H_X$ denote its pointwise stabiliser. Note that $H_X=\cap_{xH\in X} xH{x^{-1}}$, $\forall x\in G$. The following characterisation of finite height is straightforward.

\begin{lem}\label{lem:h}  Let $H$ be a subgroup of $G$. Then  $height_G(H)=n<\infty$ if and only if $n+1$ is the least integer such that for any finite  subset $X\subseteq G/H$ of cardinality at least $n+1$, the pointwise stabiliser $H_X$ is finite.
\end{lem}

For our purposes, it is convenient to generalise the finite height property as follows.

\begin{deff} 
 Let $H$ be a subgroup of  a finitely generated group $G$. The splitting height of $H$ in $G$, denoted $s$-$height_G(H)$, is the least integer $n$ such that for any $k>n$ distinct cosets $g_iH$, the intersection $\cap_{i=1}^{k} g_i H{g_i^{-1}}$ is splitting-compatible in $G$. If such  $n$ exists, then $H$ is said to have {\it finite splitting height} equal to $n$ in $G$. 
\end{deff} 
The following reformulation of the splitting height will be important in the proof of our main result.

\begin{lem}\label{lem:sh}  Let $H$ be a subgroup of $G$. Then  $s$-$height_G(H)=n<\infty$ if and only if $n$ is the least integer such that for any finite subset $X\subseteq G/H$ of cardinality at least $n+1$, the pointwise stabiliser $H_X$ is splitting-compatible.
\end{lem}
\begin{proof}The splitting height satisfies $s$-$height_G(H)=n<\infty$ if and only if for any $k>n$ distinct cosets $g_iH$, the intersection $H_X=\cap_{i=1}^{k} g_iH{g_i^{-1}}$ is splitting-compatible in $G$ where $X=\{g_1H, \dots, g_kH\}\subseteq G/H$. The latter condition is equivalent to the pointwise stabiliser $H_X$ of $X$ being splitting-compatible whenever $n< |X|<\infty$.
\end{proof}

The next result allows one to construct new examples of subgroups with finite splitting height.

\begin{prop}\label{prop:ext}  Given a short exact sequence of groups $1\to N\to G\xrightarrow{\pi} Q$, suppose $H$ is a subgroup of $Q$ of finite height $n$. Then $\overline{H}=\pi^{-1}(H)$ has splitting height at most $n$.
\end{prop}
\begin{proof} Note that the natural bijection between $G/{\overline{H}}$ and $Q/H$ is equivariant with respect to the projection $\pi:G\to Q$. For any finite $X\subseteq Q/H$ of cardinality at least $n+1$, the pointwise stabiliser $H_X\subseteq Q$ is finite by Lemma \ref{lem:h}. This implies that for any finite $X\subseteq G/{\overline{H}}$ of cardinality at least $n+1$, the pointwise stabiliser $\overline{H}_X\subseteq G$ is a finite extension of $N$ and hence commensurable with $N$. This shows that $\overline{H}_X$ is commensurated in $G$ and hence is splitting-compatible by Proposition \ref{prop:comm}.\end{proof}

Let $G$ be a finitely generated group. Subgroups of $G$ of finite height and hence strongly quasiconvex  subgroups are  examples of groups with finite splitting height. 
The next example comes from applying Proposition \ref{prop:ext} to the class of  convex cocompact subgroups in mapping class groups.

\begin{example} Let ${\rm Mod}(S)$ be the mapping class group of a closed oriented surface $S$ of genus at least 2. As it is pointed out in Corollary 1.2 of \cite{AMST}, a convex cocompact subgroup in ${\rm Mod}(S)$  is stable and hence has finite height. A {\it Schottky subgroup} in ${\rm Mod}(S)$ is a convex cocompact subgroup which is free of finite rank. Let $F$ be a finite rank free subgroup of ${\rm Mod}(S)$. By a theorem of Farb and Mosher \cite{FaMo}, $\pi_1(S)\rtimes F$ is hyperbolic if and only if $F$ is a Schottky subgroup. The group $\pi_1(S)\rtimes F$ naturally sits in  ${\rm Mod}(S, p)$ - the mapping class group of $S$ punctured at the base point $p$, which fits into a short exact sequence 
$$1\to \pi_1(S)\to {\rm Mod}(S, p)\xrightarrow{q} {\rm Mod}(S)\to 1.$$
Applying Proposition \ref{prop:ext}, it follows that if  $\pi_1(S)\rtimes F$ is  hyperbolic, then it has finite splitting height in ${\rm Mod}(S, p)$.  More generally, if $C$ is a convex compact subgroup of ${\rm Mod}(S)$, then $q^{-1}(C)$ has finite splitting height in ${\rm Mod}(S, p)$.
\end{example}

\begin{example}\label{ex:rel} Let $G$ be a relatively hyperbolic group.   Recall that a subgroup of $G$ is called {\it parabolic} if it can be conjugated into a peripheral subgroup of $G$. 

Dahmani defined a subgroup of $H$ of  $G$ to be {\it quasiconvex relative} to the peripheral subgroups if the action of $G$ as a geometrically finite convergence group on its Bowditch boundary restricts to a geometrically finite convergence group action of $H$ on its limit set \cite{Dah}, \cite[\S 6]{HW2}.

By Theorem 1.4 of \cite{HW1}, if $H$ is relatively quasiconvex in $G$, then there exists an integer $n>0$ so that for any $k>n$ and distinct cosets $g_iH$, the intersection $\cap_{i=1}^k g_iH{g_i^{-1}}$  is either finite or parabolic. So, if every parabolic subgroup of $G$ is splitting-compatible, for example when peripheral subgroups are virtually polycyclic, this gives us that $H$ has finite splitting height at most $n$ in $G$. 
\end{example}

\section{Finite splitting height implies splitting-compatible}\label{sec:main}

Let $H$ be a subgroup of  $G$.  We will show that if $H$ has  finite splitting height in $G$, then  it is splitting-compatible in $G$. 

We recall (see \cite[Definition 1]{Niblo}) that for an $H$-proper $H$-almost invariant subset $A\subset G$ satisfying $HA=A$, the {\it splitting obstruction} of the triple $(G, H, A)$ is defined as
$$S_A(G, H)=\{g\in G \; |\; A\cap gA\ne \emptyset, A\cap gA^*\ne \emptyset, A^*\cap gA\ne \emptyset, A^*\cap gA^*\ne \emptyset\}.$$
So  $g\in S_A(G, H)$ if and only if  $gA$ is not nested with $A$. 
Given two hyperplanes $J=(A, A^*)$ and $gJ=(gA, gA^*)$ of $X_A$, it follows that $J$ and $gJ$ are transverse if and only if $g \in S_A(G, H)$. 

Clearly $S_A(G, H)H=S_A(G, H)$. Denote by $\pi_H: G\to G/H$ the quotient map of the right action of $H$ on $G$.

\begin{lem}\label{lem:compact} Let $H$ be a subgroup of a finitely generated group $G$ and $A$ be an $H$-proper $H$-almost invariant subset of $G$ satisfying $HA=A$. If $\pi_H(S_A(G, H))$ is  finite, then  $G$ splits nontrivially over a subgroup commensurable with $H$. 
\end{lem}
\begin{proof} Consider the dual cube complex $X_A$ and the hyperplane $J=(A, A^*)$. Since  $\pi_H(S_A(G, H))$  is  finite, there can only be finitely many hyperplanes that are transverse to $J$. Since every edge $e$ of the carrier of $J$ exits some $B\in \Sigma_A$ \cite[Lemma 3.9]{S1}, it follows that every such edge defines a unique  hyperplane $J_e$ that intersects $J$ (see Figure \ref{fig2}). But since there are finitely many hyperplanes that are transverse to $J$, this implies that the carrier of $J$ and hence $J$ are compact.  Then arguing as in the proof of Proposition \ref{prop:comm} by considering the 1-skeleton of $X_A$, we obtain an action with unbounded orbits on a graph with more than one end. The result now follows from Theorem 4.1 of \cite{Dun:cut}.
\end{proof}
\begin{figure}
\includegraphics[scale=0.5]{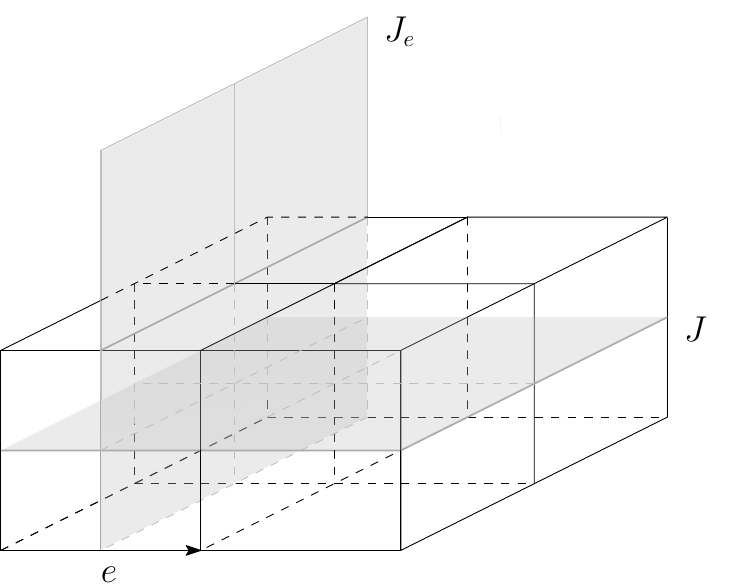}
\caption{The hyperplane $J$, its carrier and the edge $e$ defining the hyperplane $J_e$. Both hyperplanes are shaded.}
\label{fig2}
\end{figure} 
\begin{remark} In \cite[Theorem B]{Niblo}, Niblo showed that if $H$ is finitely generated and  $\langle S_A(G, H)\rangle \leq Comm_G(H)$, then $G$ splits nontrivially over a subgroup commensurable with $H$.  There, the finite generation of $H$ is used to ensure that the coboundary of $A$ is connected in a suitable Cayley graph of $G$. This is a key step in proving that the hyperplanes of the dual cube complex $X_A$ are compact \cite[Proposition 7]{Niblo}. In view of Proposition \ref{prop:comm}, we pose the following question.
\end{remark}

\begin{ques} Suppose $H$ is a  subgroup of a finitely generated group $G$ and $A$ is $H$-proper $H$-almost invariant subset satisfying $HAH=A$. Suppose  $\langle S_A(G, H)\rangle \leq Comm_G(H)$. Are the panellings of the Cayley complex $C_A$ compact?
\end{ques}

In what follows, we use Kropholler's corner argument \cite[Lemma 4.17]{Krop} to deal with the splitting obstruction when passing to intersections of conjugates of $H$.

\begin{lem}[Kropholler's corner]\label{lem:corner} Suppose $A$ and $B$ are $H$-proper $H$-almost invariant subsets of $G$  such that $AH=A$, $BH=B$. If $g\in A^*$ and $g^{-1}\in B^*$, then  $A\cap gB$ is $H\cap gHg^{-1}$-almost-invariant. 
\end{lem}

Next proposition is a reformulation of Theorem \ref{thmA}. 

\begin{prop}\label{propFSH} Let $G$ be a finitely generated group and $H$ be a subgroup with $s$-$height_G(H)=n<\infty$. Then $H$ is splitting-compatible in $G$.
\end{prop}
\begin{proof} By a way of contradiction, suppose $G$ does not split nontrivially over a subgroup commensurable with a subgroup of $H$. Suppose $A$ is an $H$-proper $H$-almost invariant subset satisfying $HAH=A$.

Define $H_1=H$, $A_1=A$ and $X_1=\{H\}\subset G/H$ and assume by induction that for $j>0$, we have constructed  finite subsets  $X_i\subset G/H$  such that $|X_1|\lneq \dots \lneq |X_j|$, subgroups $H_i=H_{X_i}$ of $H$ and  $H_i$-proper $H_i$-almost invariant subsets $A_i$ satisfying $H_iA_iH=A_i$ for all  $1\leq i\leq j$.

Now, consider the splitting obstruction $S_{A_j}(G, H_j)$ of the triple $(G, H_j, A_j)$. The set $\pi_{H_j}(S_{A_j}(G, H_j))$ must be infinite, for otherwise by Lemma \ref{lem:compact}, $G$ will split nontrivially over a subgroup commensurable with $H_j$. Since $H_j$ has finite index in $Stab(X_j)$, this implies that there exists  $g\in S_{A_j}(G, H_j)$ such that $g\notin Stab(X_j)$. Let $K=H_j\cap gH_j{g^{-1}}$. By Kropholler's corner Lemma \ref{lem:corner}, one of the corners 
$$A_j\cap gA_j, \; A_j\cap gA_j^*, \; A_j^*\cap gA_j, \; A_j^*\cap gA_j^*$$
is a $K$-almost-invariant set. Denote this corner by $B$. 

First, assume that either  $B\subseteq KF$ or $B^*\subseteq KF$ where $F$ is a finite subset of  $G$.  Setting $C=B$ or $B^*$, in either case we have $CH=C\subseteq KF$. Let $x\in C$, then we have 
$$xH\subseteq KF \;\; \Rightarrow \;\; H\subseteq x^{-1}Kx E, \; |E|<\infty.$$
This implies $[H: H\cap x^{-1}Kx]<\infty.$ Denote  $X_{j+1}=\{H\}\cup x^{-1}(X_j\cup g X_j)$ and $H_{j+1}=H_{X_{j+1}}=H\cap x^{-1}Kx$. Note $|X_j|\lneq |X_{j+1}|$, because $g\notin Stab(X_j)$. Since $H_{j+1}$ has finite index in $H$, the set $A_{j+1}=A$ is a $H_{j+1}$-proper $H_{j+1}$-almost-invariant set satisfying $H_{j+1}A_{j+1}H=A_{j+1}$. 

Now, suppose neither $B$ nor $B^*$ is $K$-finite, then $B$ is a $K$-proper $K$-almost-invariant set in which case we denote $X_{j+1}=X_j\cup g X_j$, $H_{j+1}=H_{X_{j+1}}=K$ and $A_{j+1}=B$.  Then again $A_{j+1}$ is a $H_{j+1}$-proper $H_{j+1}$-almost-invariant such that $H_{j+1}A_{j+1}H=A_{j+1}$ and $|X_j|\lneq |X_{j+1}|$. This finishes the induction on $j$.

 In this way, we construct  finite  subsets $X_i\subset G/H$ with $|X_i|\lneq |X_{i+1}|$, a sequence of subgroups $H_i=H_{X_i}$ of $H$,  and an $H_i$-proper $H_i$-almost-invariant subsets $A_i$ satisfying $H_iA_iH=A_i$ for all $i>0$. By Lemma \ref{lem:sh}, we get that $H_{n+1}$ is splitting-compatible in $G$. Applying this to the triple $(G, H_{n+1}, A_{n+1})$ gives us a contradiction.
\end{proof}

\section{Applications}

In this section, we give some applications of Theorem \ref{thmA}.

 \begin{cor}\label{cor:ext}  Given a short exact sequence of groups $1\to N\to G\xrightarrow{\pi} Q$, suppose $H$ is a subgroup of $Q$ of finite height $n$. Then $\overline{H}=\pi^{-1}(H)$ is splitting-compatible in $G$.
\end{cor}
\begin{proof} The result follows from Propositions \ref{prop:ext} and \ref{propFSH}.\end{proof}

\begin{cor}\label{cor_all} Let $G$ be a finitely generated group. 
\begin{enumerate}[label=($\roman*$)]
\item If $H$ is a subgroup of finite height in $G$, then Conjecture \ref{conj:KR} holds.
\item If $H$ is a relatively quasiconvex subgroup of a relatively hyperbolic group $G$ with splitting-compatible parabolic subgroups, then Conjecture \ref{conj:KR} holds.
\item If $H$ is a relatively quasiconvex subgroup of a relatively hyperbolic group $G$ with virtually polycyclic peripheral subgroups, then Conjecture \ref{conj:KR} holds.
\end{enumerate}
\end{cor}
\begin{proof}If $H$ has finite height, then $H$ has finite splitting height in $G$ and we can apply Theorem \ref{thmA} to deduce part (i). 

For (ii),  by Theorem 1.4 of \cite{HW1} (cf.~Example \ref{ex:rel}), $H$ has finite splitting height in $G$.  Again applying Theorem \ref{thmA} gives (ii). 

Since virtually polycyclic groups are splitting-compatible, part (iii) is a direct application of part (ii).
\end{proof}

\begin{cor}\label{cor_T} Let $G$ be a finitely generated group. Suppose  $H$ is a finite-height codimension-1 subgroup.  If every action of $H$ on a CAT(0) cube complex has a global fixed point, in particular, if $H$ satisfies Kazhdan's property (T), then $G$ splits over a subgroup commensurable with a subgroup of $H$.
\end{cor}
\begin{proof} Let $A$ be an $H$-proper $H$-almost invariant subset of $G$ satisfying $HA=A$ and let $X_A$ be its dual cube complex. Let $J$ denote the oriented hyperplane stabilised by $H$. Since $J$ is a CAT(0) cube complex, by our assumption $H$  fixes a point on $J$. Then by Lemma 2.5 of \cite{S2}, there is an $H$-proper $H$-almost invariant subset $B$ satisfying $HBH=B$. Corollary \ref{cor_all} (i) then implies that $G$ splits over a subgroup commensurable with a subgroup of $H$.
\end{proof}

Let $G$ be a countable group with a proper, left invariant metric. A subgroup $H$ of $G$, has {\it bounded packing} in $G$ (with respect to the given metric) if for any given constant $D\geq 0$, there is a bound on the number of distinct right $H$-cosets in $G$ that are pairwise distance at most $D$ apart \cite[Definition 2.1]{HW1}. We do not need to delve into the definition of bounded packing. For our purposes, it suffices to know that the bounded packing of $H$ in $G$ is independent of the choice of proper, left invariant metric on $G$  \cite[Lemma 2.2]{HW1} and that the following properties hold.

\begin{prop}[{\cite[Proposition 2.7]{HW1}}]\label{propHW} Let $G$ be a countable group. Suppose $H$ and $K$ are subgroups with bounded packing in $G$. Then $H\cap K$ has bounded packing in $H$, $K$ and $G$.\end{prop}

\begin{prop}[\cite{Yang}]\label{propY} Let $G$ be a virtually polycyclic group. Then any subgroup $H$ has bounded packing in $G$.\end{prop}

 We will also need the following results of Sageev \cite{S2} (see also {\cite[Corollary 3.1]{HW1}}) and Hruska-Wise \cite{HW1}.

\begin{prop}[\cite{S2}]\label{prop:pack} Let $G$ be a finitely generated group and suppose $H$ is a finitely generated codimension-$1$ subgroup with bounded packing in $G$. Then there exists an $H$-proper $H$-almost invariant subset $A$ with $HA=A$ so that the dual cube complex $X_A$ is finite dimensional.
\end{prop}

\begin{prop}[{\cite[Theorem 1.1]{HW1}}]\label{propHW} Let $H$ be a relatively quasiconvex subgroup of a relatively hyperbolic group $G$. Suppose $H\cap gPg^{-1}$ has bounded packing in $gPg^{-1}$ for each conjugate of each peripheral subgroup $P$. Then $H$ has bounded packing in $G$.
\end{prop}

We are now ready to prove the last two applications. 

\begin{cor}\label{cor:rel} Let $G$ be a relatively hyperbolic group with virtually polycyclic peripheral subgroups.  Suppose  $H$ is a codimension-1,   relatively quasiconvex subgroup of $G$. Then either $G$ splits nontrivially over a subgroup commensurable with a subgroup of $H$ or there exist  
$$H=H_1> H_2>\dots >H_{k+1},$$
where $g_i\in G$, $H_{i+1}=H_i\cap g_iH_ig_i^{-1}$ is a codimension-$1$, relatively quasiconvex subgroup of $H_{i}$, and $H_{k}$ splits nontrivially over a subgroup commensurable with a subgroup of $H_{k+1}$.\end{cor}

\begin{proof}   By Proposition \ref{propY} and Proposition \ref{propHW}, we know that $H$ has bounded packing in $G$. Let $A$ be an $H$-proper $H$-almost invariant subset of $G$ satisfying $HA=A$ and let $X_A$ be its dual cube complex. By Proposition \ref{prop:pack}, we can take $X_A$ to be finite dimensional, say of dimension $n$. Then the width of $\Sigma_A$ is  $n$. Let $J$ denote the oriented hyperplane stabilised by $H$. If $H$  fixes a point on $J$,  then Lemma 2.5 of \cite{S2} implies that there exists an $H$-proper $H$-almost invariant subset $B$ satisfying $HBH=B$. Since by Corollary \ref{cor_all}, $H$ is splitting-compatible, this shows that $G$ splits over a subgroup commensurable with a subgroup of $H$.

 If $H$ does not fix a point on $J$, it contains a codimension-$1$ subgroup $H_2=H\cap g_1H{g_1^{-1}}$ for some $g_1\in G$.  Define  $A_2=\{h\in H \; | \;  hv\in J^{+}\}$ for some fixed vertex $v\in J$. Then $A_2$ is an $H_2$-proper $H_2$-almost invariant subset of $H$. Since $\dim (J)=n-1$, the width of $\Sigma_{A_2}$ is at most $n-1$. The dual cube complex  $X_{A_2}$ constructed for $(H, H_2, A_2)$ is then finite dimensional of dimension at most $n-1$. Since any virtually polycyclic group is slender, by Corollary 9.2 of \cite{Hru}, $H$ is finitely generated. Using the stability of relative quasiconvexity under finite intersections and passing to subgroups \cite[Theorem 1.2, Corollary 9.3]{Hru}, we have that $H_2$ is a relatively quasiconvex subgroup of $H$. We can then repeat the above argument replacing $(G, H, A, \Sigma_A, X_A)$ with  $(H, H_2, A_2, \Sigma_{A_2}, X_{A_2})$.  After finite iterations, say $k$ of them, either $H_{k}$ will split over a subgroup commensurable with a subgroup of $H_{k+1}$ or $\Sigma_{A_{k+1}}$ will have width $1$. In which case, $X_{A_{k+1}}$ will be a tree and thus $H_{k}$ will split nontrivially over a subgroup commensurable with  $H_{k+1}$.
\end{proof}

\begin{remark} By Theorem 1.1 of \cite{HW2}, it follows that if $H$ is a codimension-1, relatively quasiconvex subgroup of a relatively hyperbolic group $G$, then $G$ acts relatively cocompactly on the associated dual cube complex $X$. This result can be applied to Corollary \ref{cor:rel} to insure in addition that for each $i\geq 0$, the group $H_i$ acts relatively cocompactly on the associated dual cube complex. 
\end{remark}

\begin{cor}\label{cor:str}  Let $G$ be a finitely generated group. Suppose  $H$ is a codimension-1,   strongly quasiconvex subgroup of $G$. Then either $G$ splits nontrivially over a subgroup commensurable with a subgroup of $H$ or there exist  
$$H=H_1> H_2>\dots >H_{k+1},$$
where $g_i\in G$, $H_{i+1}=H_i\cap g_iH_ig_i^{-1}$  is a codimension-$1$, strongly quasiconvex subgroup of $H_{i}$, and $H_{k}$ splits nontrivially over a subgroup commensurable with a subgroup of $H_{k+1}$.\end{cor}

\begin{proof} By Theorem 1.2 of \cite{Tran}, strongly quasiconvex subgroups are finitely generated, have bounded packing and the intersection of any two strongly quasiconvex subgroups $H_1$ and $H_2$ is again strongly quasiconvex in $H_1$, $H_2$ and $G$. The proof now is completely analogous to the proof of Corollary \ref{cor:rel}.
\end{proof}
\noindent Corollary \ref{cor:str} can be viewed as an extension of \cite[Corollary 1.3]{AMST} and \cite[Corollary 1.3]{Tran} where we use finite dimensionality of the dual cube complex to obtain the sequence of subgroups.

\end{document}